\documentclass[10pt]{amsart}

  \baselineskip 1 mm

\usepackage{amsmath,amsthm,amsfonts,latexsym,amsopn,verbatim,amscd,amssymb}
\usepackage{hyperref}

\theoremstyle{plain}
\newtheorem{theorem}{Theorem}[section]
\newtheorem{lemma}[theorem]{Lemma}
\newtheorem{proposition}[theorem]{Proposition}
\newtheorem{corollary}[theorem]{Corollary}
\theoremstyle{definition}

\numberwithin{equation}{section}

\DeclareMathOperator{\spec}{Spec}
\DeclareMathOperator{\L-spec}{L-Spec}
\DeclareMathOperator{\Q-spec}{Q-Spec}
\DeclareMathOperator{\cent}{Cent}

\newcommand{\bnum}{\begin{enumerate}}
\newcommand{\enum}{\end{enumerate}}

\begin{document}

\title{On super integral groups}
\author[J. Dutta and R. K. Nath]{Jutirekha Dutta  and Rajat Kanti Nath$^*$}
\address{Department of Mathematical Sciences, Tezpur
University,  Napaam-784028, Sonitpur, Assam, India.
}
\email{jutirekhadutta@yahoo.com, rajatkantinath@yahoo.com$^*$}


%

\thanks{*Corresponding author}

%
%
\begin{abstract}
A finite non-abelian group $G$ is called super integral if the spectrum, Laplacian spectrum and signless Laplacian spectrum of its commuting graph contain only integers. In this paper, we first compute various spectra of  several families of finite non-abelian groups and conclude that those groups are super integral. As an application  of our results we obtain some positive integers $n$ such that  $n$-centralizer groups are super integral. We also obtain some positive rational numbers $r$ such that $G$ is super integral if it has commutativity degree $r$. In the last section, we show that  $G$ is super integral if $G$ is not isomorphic to $S_4$ and  its commuting graph is planar. We conclude the paper showing that  $G$ is super integral if its commuting graph is toroidal.   
\end{abstract}

\maketitle

%
%

\section{Introduction} \label{S:intro}

Let $A({\mathcal{G}})$ and $D({\mathcal{G}})$ denote the adjacency matrix  and degree matrix of a graph  ${\mathcal{G}}$ respectively. Then the Laplacian matrix and   signless Laplacian matrix of ${\mathcal{G}}$ are given by  $L({\mathcal{G}})  = D({\mathcal{G}}) - A({\mathcal{G}})$ and    $Q({\mathcal{G}}) = D({\mathcal{G}}) + A({\mathcal{G}})$ respectively. We write  $\spec({\mathcal{G}})$, $\L-spec({\mathcal{G}})$ and $\Q-spec({\mathcal{G}})$ to denote the spectrum, Laplacian spectrum and Signless Laplacian spectrum of ${\mathcal{G}}$. Also, $\spec({\mathcal{G}}) = \{\alpha_1^{a_1}, \alpha_2^{a_2}, \dots, \alpha_l^{a_l}\}$, $\L-spec({\mathcal{G}}) = \{\beta_1^{b_1}, \beta_2^{b_2}, \dots, \beta_m^{b_m}\}$ and $\Q-spec({\mathcal{G}}) = \{\gamma_1^{c_1}, \gamma_2^{c_2}, \dots, \gamma_n^{c_n}\}$ where $\alpha_1,  \alpha_2, \dots, \alpha_n$ are the eigenvalues of   $A({\mathcal{G}})$ with multiplicities $a_1, a_2, \dots, a_l$;   $\beta_1,  \beta_2, \dots, \beta_m$ are the eigenvalues of  $L({\mathcal{G}})$ with multiplicities $b_1, b_2, \dots, b_m$ and  $\gamma_1, \gamma_2, \dots, \gamma_n$ are the eigenvalues of   $Q({\mathcal{G}})$ with multiplicities $c_1, c_2, \dots, c_n$ respectively.

A graph $\mathcal{G}$ is called  integral or L-integral or  Q-integral according as  $\spec({\mathcal{G}})$ or $\L-spec({\mathcal{G}})$ or $\Q-spec({\mathcal{G}})$ contains only integers.
The notion of integral graph was introduced by  Harary and  Schwenk \cite{hS74} in the year 1974. A very impressive survey on integral graphs can be found in \cite{bCrS03}.
 Ahmadi et al. noted  that   integral graphs have some interests for designing the network topology of perfect state transfer networks, see  \cite{anb09} and the references there in. L-integral graphs are also  studied extensively over the years while $Q$-integral graphs are not studied much. One may conf. \cite{Abreu08, Simic07,Kirkland07,  Merries94,Simic08} and some of the references in \cite{Kirkland07}  for several interesting results of these graphs.

Let $G$ be a finite non-abelian group with center $Z(G)$. The commuting graph of  $G$, denoted by $\Gamma_G$, is a simple undirected graph whose vertex set is $G\setminus Z(G)$, and two distinct vertices $x$ and $y$ are adjacent if and only if $xy = yx$.   Various   aspects of commuting graphs of different finite groups can be found in  \cite{amr06,iJ07,mP13,par13}. 
A finite non-abelian group $G$ is called integral or L-integral or Q-integral according as $\Gamma_G$ is  integral or L-integral or Q-integral.  Then one may  ask the following questions.

\noindent \textbf{Question 1.} Which finite non-abelian groups are integral?

\noindent \textbf{Question 2.} Which finite non-abelian groups are L-integral?

\noindent \textbf{Question 3.}  Which finite non-abelian groups are Q-integral?

 A finite non-abelian group $G$ is called super integral if $\Gamma_G$ is integral, L-integral and Q-integral. Therefore, in the line of above questions, we can also ask the following question. 

\noindent \textbf{Question 4.}  Which finite non-abelian groups are super integral?

In \cite{Dutta16,DN16}, the authors have studied the spectrum of $\Gamma_G$ for several families of finite non-abelian groups and  determined several  finite non-abelian integral groups. In  this paper,  we  study the Laplacian spectrum and signless Laplacian spectrum of commuting graphs of those groups, viz., finite groups whose central quotient is isomorphic to $Sz(2)$ (the Suzuki group of order $20$) or ${\mathbb{Z}}_p \times {\mathbb{Z}}_p$ (for any pripe $p$) or $D_{2m}$ (the dihedral group of order $2m$), quasidihedral groups, generalized quaternion groups,  general  linear groups,   some projective special linear groups,   the groups constructed by Hanaki in \cite{Han96} etc. In a separate paper \cite{dn016} we have studied various energies of the commuting graphs of these groups.
Our computations in Section 2 reveal that all the above mentioned groups are super integral. Further, we shall show that all finite  AC-groups are super integral. The rest part of this paper is devoted to some applications of the results obtained in Section 2.

Let $x$ be an element of $G$. Then centralizer of $x$ in $G$ denoted by $C_G(x)$ is the set given by $\{y \in G : xy = yx\}$. Let $\cent(G) = \{C_G(x) : x \in G\}$. A group $G$ is called an $n$-centralizer group if $|\cent(G)| = n$. The study of $n$-centralizer groups was initiated by  Belcastro and  Sherman   \cite{bG94} in the year 1994. The readers may conf. \cite{Dutta10} for various results on $n$-centralizer groups. As an application of our results, in Section 3, we determine some positive integers $n$ such that $G$ is super integral if $|\cent(G)| = n$. 

The commutativity degree of $G$ denoted by $\Pr(G)$ is the probability that a randomly chosen pair of elements of $G$ commute. Clearly, $\Pr(G) = 1$ if and only if $G$ is abelian. For a non-abelian group, it was shown in \cite{Gustafson73} that $\Pr(G) \leq 5/8$. Since then many mathematicians have studied this notion. The readers may conf. \cite{Caste10,Nath08} for various results on $\Pr(G)$. A survey of recent works on $\Pr(G)$ can be found in \cite{Dnp13}. Using our results, in Section 4, we shall show that $G$ is super integral if $\Pr(G) \in \{\frac{5}{14}, \frac{2}{5}, \frac{11}{27}, \frac{1}{2}, \frac{5}{8}\}$.

Recall that  genus of a graph is the smallest non-negative integer $n$ such that the graph can be embedded on the surface obtained by attaching $n$ handles to a sphere. A graph is said to be planar or toroidal if the genus of the graph is zero or one respectively. It is worth mentioning that  Afkhami et al. \cite{AF14} and Das et al. \cite{das13} have classified all finite non-abelian groups whose commuting graphs are planar or toroidal recently. In  the last section, we shall show that  a finite non-abelian group $G$ is super integral if it is not isomorphic to $S_4$, the symmetric group of degree $4$, and the commuting graph of $G$ is planar. We also show that  a finite non-abelian group  is  super integral if its  commuting graph    is toroidal.

\section{Some Computations}  
 It is well-known that  $\L-spec(K_n) = \{0^1, n^{n - 1}\}$ and $\Q-spec(K_n) = \{(2n -2)^1, (n - 2)^{n - 1}\}$ where $K_n$ denotes the complete graph on $n$ vertices. Further, we have the following theorem. 

\begin{theorem}\label{prethm1}
If $\mathcal{G} = l_1K_{m_1}\sqcup l_2K_{m_2}\sqcup\cdots  \sqcup l_kK_{m_k}$, where $l_iK_{m_i}$ denotes the disjoint union of $l_i$ copies of  $K_{m_i}$ for $1 \leq i \leq k$, then 
\[
\L-spec({\mathcal{G}}) = \left\{0^{\sum_{i=1}^k l_i}, m_1^{l_1(m_1 - 1)}, m_2^{l_2(m_2 - 1)}, \dots, m_k^{l_k(m_k - 1)}\right\}
\] and
\begin{align*}
\Q-spec({\mathcal{G}}) = \{&(2m_1 - 2)^{l_1}, (m_1 - 2)^{l_1(m_1 - 1)}, (2m_2 - 2)^{l_2}, (m_2 - 2)^{l_2(m_2 - 1)}, \\ 
&\dots, (2m_k - 2)^{l_k}, (m_k - 2)^{l_k(m_k - 1)}\}.
\end{align*}
\end{theorem}

 Now we compute the  Laplacian spectrum and signless Laplacian spectrum of the  commuting graphs of some  families of finite non-abelian groups. We begin with some families of groups whose central factors are some well-known  groups. 
\begin{theorem} \label{order-20}
Let $G$ be a finite group and $\frac{G}{Z(G)} \cong Sz(2)$, where $Sz(2)$ is the Suzuki group presented by $\langle a, b : a^5 = b^4 = 1, b^{-1}ab = a^2 \rangle$. Then
\begin{align*}
\L-spec(\Gamma_G) = & \{0^6, (4|Z(G)|)^{4|Z(G)| - 1}, (3|Z(G)|)^{15|Z(G)| - 5}\}
\text{\quad and}\\ 
\Q-spec(\Gamma_G) = &\{(8|Z(G)| - 2)^1, (4|Z(G)| - 2)^{4|Z(G)| - 1}, (6|Z(G)| - 2)^5, (3|Z(G)| - 2)^{15|Z(G)| - 5}\}.
\end{align*}
\end{theorem}
\begin{proof}
It was shown in \cite[Theorem 2]{Dutta16} that
$\Gamma_G = K_{4|Z(G)|}\sqcup 5K_{3|Z(G)|}$.
 Therefore, by   Theorem \ref{prethm1}, the result follows.
\end{proof}

\begin{theorem}\label{main2}
Let $G$ be a finite group such that $\frac{G}{Z(G)} \cong {\mathbb{Z}}_p \times {\mathbb{Z}}_p$, where $p$ is a prime integer. Then 
\begin{align*}
\L-spec(\Gamma_G) = &\{0^{p +1}, ((p - 1)|Z(G)|)^{(p^2 - 1)|Z(G)| - p - 1}\}
\text{ and}\\ 
\Q-spec(\Gamma_G) = &\{(2(p - 1)|Z(G)| - 2)^{p + 1}, ((p - 1)|Z(G)| - 2)^{(p^2 - 1)|Z(G)| - p - 1}\}.
\end{align*}
\end{theorem}
\begin{proof}
It was shown in \cite[Theorem 2.1]{DN16} that $\Gamma_G = (p + 1) K_{(p - 1)|Z(G)|}$. Hence the result follows from  Theorem \ref{prethm1}. 
\end{proof}
\noindent As a corollary we have the following result.
\begin{corollary}
Let $G$ be a non-abelian group of order $p^3$, for any prime $p$, then  
\begin{align*}
\L-spec(\Gamma_G) = &\{0^{p +1}, (p^2 - p)^{(p^3 - 2p  - 1}\}
\text{ and }\\
\Q-spec(\Gamma_G) = & \{(2p^2 - 2p - 2)^{p + 1}, (p^2 - p - 2)^{p^3 - 2p  - 1}\}.
\end{align*}
\end{corollary}

\begin{proof}
Note that $|Z(G)| = p$ and  $\frac{G}{Z(G)} \cong {\mathbb{Z}}_p \times {\mathbb{Z}}_p$. Hence the  result follows from Theorem \ref{main2}.
\end{proof}

\begin{theorem}\label{main4}
Let $G$ be a finite group such that $\frac{G}{Z(G)} \cong D_{2m}$, for $m \geq 2$. Then
\begin{align*}
\L-spec(\Gamma_G) = &\{0^{m + 1}, ((m - 1)|Z(G)|)^{(m - 1)|Z(G)| - 1}, (|Z(G)|)^{m(|Z(G)| - 1)}\}
\text{ and }\\
\Q-spec(\Gamma_G) = & \{(2(m - 1)|Z(G)| - 2)^1, ((m - 1)|Z(G)| - 2)^{(m - 1)|Z(G)| - 1},\\ & \hspace{4cm} (2|Z(G)| - 2)^m, (|Z(G)| - 2)^{m(|Z(G)| - 1)}\}.
\end{align*}
\end{theorem}

\begin{proof}
It was shown in \cite[Theorem 2.5]{DN16} that $\Gamma_G = K_{(m - 1)|Z(G)|} \sqcup m K_{|Z(G)|}$. Hence the result follows from  Theorem \ref{prethm1}. 
\end{proof}

\noindent Using Theorem \ref{main4}, we now compute the Laplacian and signless Laplacian spectrum of the commuting graphs of the groups $M_{2mn}, D_{2m}$ and $Q_{4n}$ respectively.   

\begin{corollary}\label{main05}
Let $M_{2mn} = \langle a, b : a^m = b^{2n} = 1, bab^{-1} = a^{-1} \rangle$ be a metacyclic group, where $m > 2$. Then
\[
\L-spec(\Gamma_{M_{2mn}}) = \begin{cases}
\{0^{m + 1}, (mn - n)^{mn -n -1}, n^{mn - m}\} & \text{if $m$ is odd}\\
\{0^{\frac{m}{2} + 1}, (mn - 2n)^{mn -2n -1}, (2n)^{mn - \frac{m}{2}}\} & \text{if $m$ is even}\\
\end{cases}
\]
and
 \[\Q-spec(\Gamma_{M_{2mn}}) = \begin{cases}
\!\!\{(2mn -2n -2)^1, (mn - n - 2)^{mn - n - 1}, (2n - 2)^m, (n - 2)^{mn - m}\}  \text{ if $m$ is odd}\\
\!\!\{(2mn -4n -2)^1, (mn - 2n - 2)^{mn - 2n - 1}, (4n - 2)^{\frac{m}{2}},\\
\hfill (2n - 2)^{mn - \frac{m}{2}}\} \text{ if $m$ is even}. 
\end{cases}
\]
\end{corollary}
\begin{proof}
Observe that $Z(M_{2mn}) = \langle b^2 \rangle$ or $\langle b^2 \rangle \cup a^{\frac{m}{2}}\langle b^2 \rangle$ according as $m$ is odd or even.  Also, it is easy to see that $\frac{M_{2mn}}{Z(M_{2mn})} \cong D_{2m}$ or $D_m$ according as $m$ is odd or even. Hence, the result follows from Theorem \ref{main4}.
\end{proof}
\noindent As a corollary to the above result we have the following results.
\begin{corollary}\label{main005}
Let $D_{2m} = \langle a, b : a^m = b^{2} = 1, bab^{-1} = a^{-1} \rangle$ be  the dihedral group of order $2m$, where $m > 2$. Then 
\[
\L-spec(\Gamma_{D_{2m}}) = \begin{cases}
\{0^{m + 1}, (m - 1)^{m - 2}\} & \text{ if $m$ is odd}\\
\{0^{\frac{m}{2} + 1}, (m - 2)^{m - 3}, 2^{\frac{m}{2}}\} &\text{ if $m$ is even}\\
\end{cases}
\]
and
\[\Q-spec(\Gamma_{D_{2m}}) = \begin{cases}
\{(2m - 4)^1, (m - 3)^{m - 2}, (2n - 2)^m, 0^m\} & \text{ if $m$ is odd}\\
\{(2m  - 6)^1, (m - 4)^{m - 3}, 2^{\frac{m}{2}}, 0^\frac{m}{2}\} & \text{ if $m$ is even}. 
\end{cases}
\]
\end{corollary}

\begin{corollary}\label{q4m}
Let $Q_{4n} = \langle x, y : y^{2n} = 1, x^2 = y^n,yxy^{-1} = y^{-1}\rangle$, where $n \geq 2$, be the   generalized quaternion group of order $4n$. Then
\[
\L-spec(\Gamma_{Q_{4n}}) =  \{0^{n + 1}, (2n - 2)^{2n - 3}, 2^n\}\text{ and }
\Q-spec(\Gamma_{Q_{4n}}) =  \{(4n - 6)^1, (2n - 4)^{2n - 3}, 2^n, 0^n\}.
\]

\end{corollary}
\begin{proof}
The result follows from Theorem \ref{main4} noting that  $Z(Q_{4n}) = \{1, a^n\}$ and  $\frac{Q_{4n}}{Z(Q_{4n})} \cong D_{2n}$.   
\end{proof}

Now we compute the Laplacian spectrum and signless Laplacian spectrum of the  commuting graphs of some  well-known families of finite non-abelian groups. 

\begin{proposition}\label{order-pq}
Let $G$ be a non-abelian group of order $pq$, where $p$ and $q$ are primes with $p\mid (q - 1)$. Then
\begin{align*}
\L-spec(\Gamma_G) = & \{0^{q + 1}, (q - 1)^{q - 2}, (p - 1)^{pq - 2q}\} \quad \text{ and } \\
\Q-spec(\Gamma_G) = & \{(2q - 4)^1, (q - 3)^{q - 2}, (2p - 4)^q, (p - 3)^{pq - 2q}\}.
\end{align*}
\end{proposition}

\begin{proof}
It was shown in \cite[Lemma 3]{Dutta16} that $\Gamma_G = K_{q-1} \sqcup qK_{p - 1}$. Hence, the result follows from  Theorem  \ref{prethm1}.  
\end{proof}

\begin{proposition}\label{semid}
Let $QD_{2^n}$ denote the quasidihedral group $\langle a, b : a^{2^{n-1}} =  b^2 = 1, bab^{-1} = a^{2^{n - 2} - 1}\rangle$, where $n \geq 4$. Then 
\begin{align*}
\L-spec(\Gamma_{QD_{2^n}}) = & \{0^{2^{n - 2} + 1}, (2^{n - 1} - 2)^{2^{n - 1}-3}, 2^{2^{n - 2}}\} \text{\quad and } \\
\Q-spec(\Gamma_{QD_{2^n}}) = &\{(2^n - 6)^1, (2^{n - 1} - 4)^{{2^{n - 1} - 3}}, 2^{2^{n - 2}}, 0^{2^{n - 2}}\}.
\end{align*}
\end{proposition}
\begin{proof}
It was shown in \cite[Proposition 1]{Dutta16} that
 $\Gamma_{QD_{2^n}} = K_{2^{n - 1} - 2} \sqcup 2^{n - 2} K_2$. Hence,  the result follows from  Theorem \ref{prethm1}.
\end{proof}

\begin{proposition}\label{psl}
The Laplacian spectrum and signless Laplacian spectrum of the commuting graph of the projective special linear group  $PSL(2, 2^k)$, where $k \geq 2$,   are given by
\[
 \{0^{2^{2k} + 2^k + 1}, (2^k - 1)^{2^{2k} - 2^k - 2}, (2^k - 2)^{2^{k - 1}(2^{2k} - 2^{k + 1} - 3)}, (2^k)^{2^{k - 1}(2^{2k} - 2^{k + 1} + 1)}\} 
 \text{ and}
\] 
\begin{align*}
\{(2^{k + 1} - 4)^{2^k + 1}, (2^k - 3)^{2^{2k} - 2^k - 2}, & (2^{k + 1} - 6)^{2^{k - 1}(2^k + 1)}, (2^k - 4)^{2^{k - 1}(2^{2k} - 2^{k + 1} - 3)},\\
& (2^{k + 1} - 2)^{2^{k - 1}(2^k - 1)}, (2^k - 2)^{2^{k - 1}(2^{2k} - 2^{k + 1} + 1)}\} \text{ respectively.}
\end{align*}
\end{proposition}

\begin{proof}
It was shown in \cite[Proposition 2]{Dutta16} that 
\[
\Gamma_{PSL(2, 2^k)} = (2^k + 1)K_{2^k - 1} \sqcup 2^{k - 1}(2^k + 1)K_{2^k - 2} \sqcup 2^{k - 1}(2^k - 1)K_{2^k}.
\]
 Hence, the result follows from   Theorem \ref{prethm1}.
\end{proof}

\begin{proposition}
The Laplacian spectrum and signless Laplacian spectrum of the commuting graph of the general linear group  $GL(2, q)$, where $q = p^n > 2$ and $p$ is a prime integer,   are given by
\[
 \{0^{q^2 + q + 1}, (q^2 - 3q + 2)^{\frac{q(q + 1)(q^2 - 3q + 1)}{2}}, (q^2 - q)^{\frac{q(q - 1)(q^2 - q - 1)}{2}}, (q^2 - 2q + 1)^{q(q + 1)(q - 2)}\} \text{ and}
\]
\begin{align*}
\{(2q^2 - 6q - 2)^{\frac{q(q + 1)}{2}}, (q^2 - 3q)&^{\frac{q(q + 1)(q^2 - 3q + 1)}{2}}, (2q^2 - 2q - 2)^{\frac{q(q - 1)}{2}}, (q^2 - q - 2)^{\frac{q(q - 1)(q^2 - q - 1)}{2}},\\
& (2q^2 - 4q)^{q + 1}, (q^2 + 2q -1)^{q(q + 1)(q - 2)}\} \text{ respectively.}
\end{align*}

\end{proposition}

\begin{proof}
It was shown in \cite[Proposition 3]{Dutta16} that 
\[
\Gamma_{GL(2, q)} = \frac{q(q + 1)}{2}K_{q^2 - 3q + 2} \sqcup \frac{q(q - 1)}{2}K_{q^2 - q} \sqcup (q + 1)K_{q^2 - 2q + 1}.
\]
Hence, the result follows from  Theorem \ref{prethm1}.
\end{proof}

\begin{proposition}\label{Hanaki1}
Let $F = GF(2^n), n \geq 2$ and $\vartheta$ be the Frobenius  automorphism of $F$, i. e., $\vartheta(x) = x^2$ for all $x \in F$. Let $A(n, \vartheta)$ denote the group  
\[
 \left\lbrace U(a, b) = \begin{bmatrix}
        1 & 0 & 0\\
        a & 1 & 0\\
        b & \vartheta(a) & 1
       \end{bmatrix} : a, b \in F \right\rbrace.
\] 
under matrix multiplication given by $U(a, b)U(a', b') = U(a + a', b + b' + a'\vartheta(a))$. Then
\begin{align*}
\L-spec(\Gamma_{A(n, \vartheta)}) =& \{0^{2^n - 1}, (2^n)^{2^{2n} - 2^{n + 1} + 1}\}
 \text{ and } \\
\Q-spec(\Gamma_{A(n, \vartheta)}) = &\{(2^{n + 1} - 2)^{2^n - 1}, (2^n - 2)^{2^{2n} - 2^{n + 1} + 1}\}.
\end{align*}
\end{proposition}

\begin{proof}
It was shown in \cite[Proposition 4]{Dutta16} that
 $\Gamma_{A(n, \vartheta)} = (2^n - 1)K_{2^n}$. Hence the result follows from     Theorem \ref{prethm1}.
\end{proof}

\begin{proposition}\label{Hanaki2}
Let $F = GF(p^n)$ where $p$ is a prime. Let $A(n, p)$ denote    the  group 
\[
\left\lbrace V(a, b, c) = \begin{bmatrix}
        1 & 0 & 0\\
        a & 1 & 0\\
        b & c & 1
       \end{bmatrix} : a, b, c \in F \right\rbrace.
\]
under matrix multiplication $V(a, b, c)V(a', b', c') = V(a + a', b + b' + ca', c + c')$.
Then
\begin{align*}
\L-spec(\Gamma_{A(n, p)}) = & \{0^{p^n + 1}, (p^{2n} - p^n)^{p^{3n} -2p^{n} -  1}\}
\text{ and }\\
\Q-spec(\Gamma_{A(n, p)}) = & \{(2p^{2n} - 2p^n - 2)^{p^n + 1}, (p^{2n} - p^n - 2)^{p^{3n} -2p^{n} -  1}\}.
\end{align*}
\end{proposition}

\begin{proof}
It was shown in \cite[Proposition 1]{Dutta16} that $\Gamma_{A(n, p)} =  (p^n + 1)K_{p^{2n} - p^n}.$ 
Hence the result follows  from   Theorem \ref{prethm1}.
\end{proof}

Our computations reveal that all the groups considered above are both L-integral  and Q-integral. Also, it was shown in \cite{Dutta16,DN16} that theses groups are integral. Hence, all these groups are super integral.

A group $G$ is called an AC-group if $C_G(x)$ is abelian for all $x \in G\setminus Z(G)$.  Various aspects of AC-groups can be found in \cite{Ab06,das13,Roc75}. In \cite{Dutta16}, the authors have shown that finite AC-groups are integral. In the following two results we shall show that if $G$ is an AC-group or $G$ is isomorphic to $H\times A$, where $H$ is a finite non-abelian AC-group and $A$ is any finite abelian group then $G$ is L-integral as well as Q-integral. Hence, $G$ is super integral.

\begin{theorem}\label{AC-group}
Let $G$ be a finite non-abelian   AC-group.  Then
\begin{align*}
\L-spec(\Gamma_G)  = &\{0^n, (|X_1| - |Z(G)|)^{|X_1| - |Z(G)| - 1}, \dots, (|X_n| - |Z(G)|)^{|X_n| - |Z(G)| - 1}\} \text{ and }\\
\Q-spec(\Gamma_G) = &\{(2(|X_1| - |Z(G)|) - 2)^1, (|X_1| - |Z(G)| - 2)^{|X_1| - |Z(G)| - 1}, \dots,\\
&\hspace{1.5cm} (2(|X_n| - |Z(G)|) - 2)^1, (|X_n| - |Z(G)| - 2)^{|X_n| - |Z(G)| - 1}\}.
\end{align*}
where $X_1,\dots, X_n$ are the distinct centralizers of non-central elements of $G$.
\end{theorem}

\begin{proof}
By Lemma 1 in \cite{Dutta16}, we have $\Gamma_G =  \overset{n}{\underset{i = 1}{\sqcup}}K_{|X_i| - |Z(G)|}$.
Therefore, the result follows from Theorem \ref{prethm1}.
\end{proof}

\begin{corollary}\label{AC-cor}
Let $G \cong H\times A$ where $H$ is a finite non-abelian  AC-group and $A$ is any finite abelian group.  Then  
\begin{align*}
\L-spec(\Gamma_G)  = &\{0^n, (|A|(|X_1| - |Z(H)|))^{|A|(|X_1| - |Z(H)|) - 1}, \dots,\\
&\hspace{4.5cm} (|A|(|X_n| - |Z(H)|))^{|A|(|X_n| - |Z(H)|) - 1}\} \text{ and }\\
\Q-spec(\Gamma_G) = &\{(2|A|(|X_1| - |Z(H)|) - 2)^1, (|A|(|X_1| - |Z(H)|) - 2)^{|A|(|X_1| - |Z(H)|) - 1}, \dots,\\
&\hspace{.7cm} (2|A|(|X_n| - |Z(H)|) - 2)^1, (|A|(|X_n| - |Z(H)|) - 2)^{|A|(|X_n| - |Z(H)|) - 1}\}.
\end{align*}
where $X_1,\dots, X_n$ are the distinct centralizers of non-central elements of $H$.  
\end{corollary}

\begin{proof}
It is easy to see that $Z(H\times A) = Z(H)\times A$ and $X_1\times A,  X_2 \times A,\dots, X_n \times A$  are the distinct centralizers of non-central elements of $H\times A$.  Therefore, if $H$ is an  AC-group then $H\times A$ is also an  AC-group. Hence, the result follows from Theorem \ref{AC-group}. 
\end{proof}

\section{$n$-centralizer groups}

For a finite group $G$, the set $C_G(x) = \{y \in G : xy = yx\}$ is called the centralizer of an element $x \in G$. Let $|\cent(G)| = |\{C_G(x) : x \in G\}|$, that is the number of distinct centralizers in $G$. A group $G$ is called an $n$-centralizer group if $|\cent(G)| = n$. The study of these groups was initiated by  Belcastro and  Sherman   \cite{bG94} in the year 1994. The readers may conf. \cite{Dutta10} for various results on these groups. In this section, we ask the following question.  

\noindent \textbf{Question 5.}  Can we determine all the positive integers $n$ such that $n$-centralizer groups are super integral?

 In the next three results we shall show that  an $n$-centralizer group is super integral if $n = 4, 5$. Further, we shall show that  a $(p + 2)$-centralizer $p$-group is super integral for any prime $p$. 

\begin{theorem}\label{4-cent}
A finite $4$-centralizer group is super integral.    
\end{theorem}
\begin{proof}
Let $G$ be a finite $4$-centralizer group. Then, by  \cite[Theorem 2]{bG94}, we have  $\frac{G}{Z(G)} \cong {\mathbb{Z}}_2 \times {\mathbb{Z}}_2$. Therefore, by Theorem \ref{main2}, we have
\[
\L-spec(\Gamma_G) = \{0^{3}, (|Z(G)|)^{3|Z(G)| - 3}\}
\text{ and } 
\Q-spec(\Gamma_G) = \{(2|Z(G)| - 2)^{3}, (|Z(G)| - 2)^{3|Z(G)| - 3}\}.
\]
Also, by \cite[Theorem 3]{Dutta16}, we have $\spec(\Gamma_G) = \{(-1)^{3|Z(G)| - 3}, (|Z(G)| - 1)^{3}\}$. Hence, $G$ is super integral group. 
\end{proof}

\noindent Further, we have the following result.

\begin{corollary}
For any prime $p$,  a finite $(p+2)$-centralizer $p$-group is super integral.

\end{corollary}
\begin{proof}
Let $G$ be a finite $(p + 2)$-centralizer $p$-group. Then, by   \cite[Lemma 2.7]{ali00}, we have  $\frac{G}{Z(G)} \cong {\mathbb{Z}}_p \times {\mathbb{Z}}_p$. Therefore, by Theorem \ref{main2}, we have
\begin{align*}
\L-spec(\Gamma_G) = &\{0^{p +1}, ((p - 1)|Z(G)|)^{(p^2 - 1)|Z(G)| - p - 1}\}
\text{ and}\\ 
\Q-spec(\Gamma_G) = &\{(2(p - 1)|Z(G)| - 2)^{p + 1}, ((p - 1)|Z(G)| - 2)^{(p^2 - 1)|Z(G)| - p - 1}\}.
\end{align*}
Also, by \cite[Corollary 2.4]{DN16}, we have $\spec(\Gamma_G) = \{(-1)^{(p^2 - 1)|Z(G)| - p - 1}, ((p - 1)|Z(G)| - 1)^{p + 1}\}$. Hence, $G$ is a super integral group.
\end{proof}

\begin{theorem}\label{5-cent}
A finite $5$-centralizer  group is super integral. 
\end{theorem}
\begin{proof}
Let $G$ be a finite $5$-centralizer group. Then by  \cite[Theorem 4]{bG94} we have  $\frac{G}{Z(G)} \cong {\mathbb{Z}}_3 \times {\mathbb{Z}}_3$ or $D_6$. Now, if $\frac{G}{Z(G)} \cong {\mathbb{Z}}_3 \times {\mathbb{Z}}_3$ then  by Theorem \ref{main2} we have
$\L-spec(\Gamma_G) = \{0^4, (2|Z(G)|)^{8|Z(G)| - 4}\}$ and $\Q-spec(\Gamma_G) = \{4|Z(G)| - 2)^4, (2|Z(G)| - 2)^{8|Z(G)| - 4}\}$. If $\frac{G}{Z(G)} \cong D_6$ then by Theorem \ref{main4} we have
$\L-spec(\Gamma_G) = \{0^4, (2|Z(G)|)^{2|Z(G)| - 1}, (|Z(G)|)^{3(|Z(G)| - 1)}\}$
  and  
$\Q-spec(\Gamma_G) =   \{(4|Z(G)| - 2)^1, (2|Z(G)| - 2)^{2|Z(G)| - 1}, (2|Z(G)| - 2)^3, (|Z(G)| - 2)^{3(|Z(G)| - 1)}\}$. Therefore, $G$ is both L-integral and Q-integral. Also, in \cite[Corollary 2.6]{DN16}, it was shown that a  finite $5$-centralizer  group $G$ is integral.  Hence, $G$ is super integral.
\end{proof}

\noindent We conclude this section with the  following corollary.
\begin{corollary}
Let $G$ be a finite non-abelian group and $\{x_1, x_2, \dots, x_r\}$ be a set of pairwise non-commuting elements of $G$ having maximal size. Then $G$ is super integral if $r = 3, 4$. 

\end{corollary}
\begin{proof}
By Lemma 2.4 in \cite{ajH07}, we have that $G$ is a $4$-centralizer or a $5$-centralizer group according as  $r = 3$ or $4$. Hence the result follows from Theorem \ref{4-cent} and Theorem \ref{5-cent}.  
\end{proof}
\section{Commutativity degree and super integral group}
Let $G$ be a finite group. The commutativity degree of $G$ is given by the ratio 
\[
\Pr(G) = \frac{|\{(x, y) \in G \times G : xy = yx\}|}{|G|^2}.
\]
The origin of the commutativity degree of a finite group lies in a paper of Erd$\ddot{\rm o}$s and Tur$\acute{\rm a}$n (see \cite{Et68}). Readers may conf. \cite{Caste10,Dnp13,Nath08} for various results on $\Pr(G)$. In this section, we ask the following question.

\noindent \textbf{Question 6.}  Can we determine all the positive rational numbers $r$ such that any group $G$ with $\Pr(G) = r$ is  super integral?

The following theorems give some rational numbers $r$ such that $G$ is super integral if $\Pr(G) = r$.
\begin{theorem}
If $\Pr(G) \in \{\frac{5}{14}, \frac{2}{5}, \frac{11}{27}, \frac{1}{2}, \frac{5}{8}\}$ then $G$ is super integral.
\end{theorem}
\begin{proof}
If $\Pr(G) \in \{\frac{5}{14}, \frac{2}{5}, \frac{11}{27}, \frac{1}{2}, \frac{5}{8}\}$ then as shown in \cite[pp. 246]{Rusin79} and \cite[pp. 451]{Nath13}, we have $\frac{G}{Z(G)}$ is isomorphic to one of the groups in $\{D_{14}, D_{10}, D_8, D_6, {\mathbb{Z}}_2\times {\mathbb{Z}}_2\}$. If $\frac{G}{Z(G)}$ is isomorphic to $D_{14}, D_{10}, D_8$ or $D_6$ then by \cite[Theorem 2.5]{DN16}, we have $G$ is integral and by Theorem \ref{main4}, we have $G$ is both L-integral and Q-integral. Hence, in this case $G$ is super integral. If $\frac{G}{Z(G)}$ is isomorphic to ${\mathbb{Z}}_2\times {\mathbb{Z}}_2$ then by \cite[Theorem 3]{Dutta16}, it follows that $G$ is integral. Also, by Theorem \ref{main2} we have  $G$ is both L-integral and Q-integral. Hence,  $G$ is super integral. This completes the proof.
\end{proof}
\begin{theorem}
Let $G$ be a finite group and $p$ the smallest prime divisor of $|G|$. If $\Pr(G) = \frac{p^2 + p - 1}{p^3}$ then $G$ is super integral. 
\end{theorem}
\begin{proof}
If $\Pr(G) = \frac{p^2 + p - 1}{p^3}$ then by \cite[Theorem 3]{dM74} we have $\frac{G}{Z(G)}$ is isomorphic to ${\mathbb{Z}}_p\times {\mathbb{Z}}_p$. Now, by \cite[Theorem 2. 1]{DN16}, we have that $G$ is integral. Again, by Theorem \ref{main2}, it follows that $G$ is L-integral as well as Q-integral. Hence,  $G$ is super integral.
\end{proof}
We conclude this section by the following result.
\begin{theorem}
If $G$ is a non-solvable group with $\Pr(G) = \frac{1}{12}$ then $G$ is super integral.
\end{theorem}
\begin{proof}
By Proposition 3.3.7 in \cite{Caste10} we have that $G$ is isomorphic to $A_5 \times B$ for some abelian group $B$. Therefore $G$ is an AC-group and hence super integral.
\end{proof}
\section{More Applications}

In this section, we show that
a finite non-abelian group $G$ not isomorphic to $S_4$ is  super integral if its commuting graph  is planar. We also show that $G$ is super integral if its commuting graph   is toroidal. We begin with the following useful  lemma.
\begin{lemma}\label{order16}
Let $G$ be a group isomorphic to any of the following groups
\begin{enumerate}
\item ${\mathbb{Z}}_2 \times D_8$
\item ${\mathbb{Z}}_2 \times Q_8$
\item $M_{16}  = \langle a, b : a^8 = b^2 = 1, bab = a^5 \rangle$
\item ${\mathbb{Z}}_4 \rtimes {\mathbb{Z}}_4 = \langle a, b : a^4 = b^4 = 1, bab^{-1} = a^{-1} \rangle$
\item $D_8 * {\mathbb{Z}}_4 = \langle a, b, c : a^4 = b^2 = c^2 =  1, ab = ba, ac = ca, bc = a^2cb \rangle$
\item $SG(16, 3)  = \langle a, b : a^4 = b^4 = 1, ab = b^{-1}a^{-1}, ab^{-1} = ba^{-1}\rangle$.
\end{enumerate}
Then $\L-spec(\Gamma_G) = \{0^3, 4^9\}$ and $\Q-spec(\Gamma_G) = \{6^3, 2^9\}$.
\end{lemma}
\begin{proof}
If $G$ is isomorphic to any of the above listed  groups, then $|G| = 16$ and $|Z(G)| = 4$. Therefore, $\frac{G}{Z(G)} \cong {\mathbb{Z}}_2 \times {\mathbb{Z}}_2$. Thus the result follows from Theorem~\ref{main2}.
\end{proof}


Now we state and proof  the main results of this section. 
\begin{theorem}
Let $\Gamma_G$ be the commuting graph of a finite non-abelian  group $G$.  If $G$ is not isomorphic to $S_4$ and $\Gamma_G$  is planar   then  $G$ is super integral. 
\end{theorem}

\begin{proof}
It was shown in \cite[Theorem 4]{Dutta16} that $G$ is an integral group if it is not isomorphic to $S_4$ and $\Gamma_G$  is planar.

 By  \cite[Theorem 2.2]{AF14}, we have that $\Gamma_G$ is planar if and only if $G$ is isomorphic to either $D_6, D_8, D_{10}$, $D_{12}, Q_8, Q_{12}, {\mathbb{Z}}_2 \times D_8, {\mathbb{Z}}_2 \times Q_8, M_{16}, {\mathbb{Z}}_4 \rtimes {\mathbb{Z}}_4, D_8 * {\mathbb{Z}}_4, SG(16, 3), A_4,$ $A_5, S_4, SL(2, 3)$ or $Sz(2)$. 

If $G \cong D_6, D_8, D_{10}$ or $D_{12}$ then by Corollary \ref{main005}, one may conclude that  $\Gamma_G$ is both L-integral and Q-integral. Hence, $G$ is  both L-integral Q-integral. If $G \cong Q_8$ or $Q_{12}$ then $G$ is both L-integral  and Q-integral,  by Corollary \ref{q4m}. If $G \cong {\mathbb{Z}}_2 \times D_8, {\mathbb{Z}}_2 \times Q_8, M_{16}, {\mathbb{Z}}_4 \rtimes {\mathbb{Z}}_4, D_8 * {\mathbb{Z}}_4$ or $SG(16, 3)$ then, by Lemma \ref{order16}, $G$ is L-integral as well as Q-integral.

If $G \cong A_4 = \langle a, b : a^2 = b^3 = (ab)^3 = 1\rangle$ then the distinct centralizers of non-central elements of $G$ are $C_{G}(a) = \{1, a, bab^2, b^2ab\}, C_{G}(b) =\{1, b, b^2\}$, $C_{G}(ab) = \{1, ab, b^2a\}, C_{G}(ba) = \{1, ba, ab^2\}$ and $C_{G}(aba) = \{1, aba, bab\}$. Note that these centralizers are abelian subgroups of $G$. Therefore, $\Gamma_{G} = K_3 \sqcup 4K_2$.  Using Theorem \ref{prethm1}, we have  
$\L-spec(\Gamma_{G}) = \{0^5, 3^2, 2^4\}$   and $\Q-spec(\Gamma_{G}) = \{4^1, 1^2, 2^4, 0^4\}$. Hence, $G$ is both L-integral and Q-integral.

If $G \cong Sz(2)$ then by Theorem \ref{order-20}, we have
$
\L-spec(\Gamma_G) = \{0^6, 4^3, 3^{10}\} \text{ and } \Q-spec(\Gamma_{G}) = \{6^1,  2^3, 4^5, 1^{10}\}$. Hence, $G$ is both L-integral and Q-integral.

If $G$ is isomorphic to $SL(2, 3)$ then it was shown in the proof of \cite[Theorem 4]{Dutta16} that
$\Gamma_G = 3K_2 \sqcup 4K_4$.  Therefore, by Theorem \ref{prethm1}, we have
$\L-spec(\Gamma_G) = \{0^7, 2^3, 4^{12}\}$ and $\Q-spec(\Gamma_{G}) = \{0^3, 2^{15}, 6^4\}$. Hence, $G$ is both L-integral and Q-integral.


If $G \cong A_5$ then by Proposition \ref{psl}, we have
\[
\L-spec(\Gamma_G) = \{0^{21}, 3^{10}, 2^{10}, 4^{18}\} \text{ and } \Q-spec(\Gamma_{G}) = \{4^5,  1^{10}, 2^{10}, 0^{10}, 6^6, 2^{18}\},
\] 
noting that $PSL(2, 4) \cong A_5$.  Hence, $G$ is both L-integral and Q-integral.

Finally, if $G \cong S_4$ then
it can be seen that the  characteristic polynomial of $L(\Gamma_G)$ is $x^5(x - 1)^3(x - 2)^4(x - 3)^6(x - 5)(x^2 - 8x + 3)^2$ and so 
\[
\L-spec(\Gamma_G) = \left\lbrace 0^5, 1^3, 2^4, 3^6, 5^1 (4 + \sqrt{13})^2, (4 - \sqrt{13})^2 \right\rbrace.
\]
Also,  the  characteristic polynomial of $Q(\Gamma_G)$ is $x^4(x - 1)^6(x - 2)^4(x - 3)^3(x^2 - 11x + 20)(x^2 - 8x + 11)^2$ and so 
\[
\Q-spec(\Gamma_G) = \left\lbrace 0^4, 1^6, 2^4, 3^3, 5^1 (4 + \sqrt{5})^2, (4 - \sqrt{5})^2,\left(\frac{11 + \sqrt{41}}{2}\right)^1, \left(\frac{11 - \sqrt{41}}{2}\right)^1 \right\rbrace.
\]
This shows that if $G \cong S_4$ then   it is neither  L-integral nor Q-integral. This completes the proof.
\end{proof}


\begin{theorem}
Let $\Gamma_G$ be the commuting graph of a finite non-abelian  group $G$.  Then $G$ is super integral if  $\Gamma_G$    is toroidal. 
\end{theorem}
\begin{proof}
By Theorem 6.6 of \cite{das13},  we have   $\Gamma_G$ is toroidal if and only if $G$ is isomorphic to either $D_{14}, D_{16}$, $Q_{16}, QD_{16},   D_6 \times {\mathbb{Z}}_3,   A_4 \times {\mathbb{Z}}_2$ or ${\mathbb{Z}}_7 \rtimes {\mathbb{Z}}_3$.

If $G \cong D_{14}$ or $D_{16}$ then, by Corollary \ref{main005}, one may conclude that  $G$ is both L-integral and Q-integral. If $G \cong Q_{16}$ then,  by Corollary \ref{q4m},    $G$ becomes both L-integral and Q-integral. If $G \cong QD_{16}$ then, by Proposition \ref{semid}, $G$ becomes both L-integral and Q-integral. If 
$G \cong {\mathbb{Z}}_7 \rtimes {\mathbb{Z}}_3$ then $G$ is both L-integral and Q-integral, by Proposition \ref{order-pq}.  If $G$ is isomorphic to  $D_6\times {\mathbb{Z}}_3$ or $A_4\times {\mathbb{Z}}_2$ then $G$ becomes both L-integral and Q-integral by Corollary \ref{AC-cor}, since   $D_6$ and $A_4$ are AC-groups. Further, it was shown in \cite[Theorem 5]{Dutta16} that $G$ is integral if $\Gamma_G$ is toroidal.  Hence, $G$ is super integral if $\Gamma_G$ is toroidal. 
\end{proof}

We  conclude the paper with the following result.

\begin{proposition}
Let $\Gamma_G$ be the commuting graph of a finite non-abelian  group $G$. Then  $\Gamma_G$ is super integral if the complement of  $\Gamma_G$  is planar.
\end{proposition}
\begin{proof}
If  the complement of  $\Gamma_G$  is planar then, by   \cite[Proposition 2.3]{Ab06}, we have  $G$ is isomorphic to either $D_6, D_8$ or $Q_8$. If $G \cong D_6$ or $D_8$ then, by Corollary \ref{main005} and \cite[Proposition 6 ]{Dutta16},  we have that $G$ is super integral.  If $G \cong Q_8$ then,  by Corollary \ref{q4m} and \cite[Proposition 7 ]{Dutta16},   it follows that $G$ is super integral. This completes the proof.
\end{proof}






\end{document}